\documentclass[12pt]{amsart}
\usepackage[utf8]{inputenc}
\usepackage{amsmath}
\usepackage{amssymb}
\usepackage{amsfonts}
\usepackage{amsthm}
\usepackage{mathrsfs} 
\usepackage{bm}
\usepackage{bbm}
\usepackage{tikz}
\usepackage{tikz-cd}
\usepackage{centernot}
\usepackage{hyperref}
\usepackage{thmtools, thm-restate}
\usepackage[margin=1.0in]{geometry}
\DeclareMathOperator{\Id}{Id}

\newcommand{\R}{\Bbb{R}}

\newcommand{\GG}{\mathcal{G}}

\newcommand{\PP}{\mathcal{P}}

\renewcommand{\SS}{\mathcal{S}}

\newcommand{\ep}{\epsilon}

\hypersetup{
    colorlinks=true,
    linkcolor=blue,
    filecolor=magenta,
    urlcolor=blue,
    citecolor=blue
}

\newtheorem{thm}{Theorem}

\newtheorem{result}{Result}[section]
\newtheorem{lem}[result]{Lemma}

\newtheorem{prp}[result]{Proposition}

\theoremstyle{definition}

\newtheorem{rmk}[result]{Remark}

\newtheorem*{ack}{Acknowledgements}

\usepackage{caption}
\usepackage{booktabs}
\usepackage{makecell}
\usepackage{siunitx}
\usepackage{tcolorbox}

\theoremstyle{remark}

\newcommand{\hide}[1]{}

\newcommand{\rough}[1]{}
\definecolor{darkgreen}{RGB}{75,150,75}
\newcommand{\review}[1]{}

\newcommand{\hides}[1]{}
\newcommand{\pub}[1]{}

\usepackage{float}
\usepackage{subfig}

\title{A new upper bound to (a variant of) the pancake problem}
\author{Zach Hunter}
\address{Mathematical Institute, Oxford University}
\email{zachary.hunter@exeter.ox.ac.uk}
\date{\today}

\begin{document}

\maketitle

\begin{abstract}
    The ``pancake problem'' asks how many prefix reversals are sufficient to sort any permutation $\pi \in \mathcal{S}_k$ to the identity. We write $f(k)$ to denote this quantity.
    
    The best known bounds are that $\frac{15}{14}k -O(1) \le f(k)\le \frac{18}{11}k+O(1)$. The proof of the upper bound is computer-assisted, and considers thousands of cases.
    
    We consider $h(k)$, how many prefix and suffix reversals are sufficient to sort any $\pi \in \mathcal{S}_k$. We observe that $\frac{15}{14}k -O(1)\le h(k)$ still holds, and give a human proof that $h(k) \le \frac{3}{2}k +O(1)$. The constant ``$\frac{3}{2}$'' is a natural barrier for the pancake problem and this variant, hence new techniques will be required to do better.
\end{abstract}

\section{Introduction}

Given a positive integer $k$, we define the pancake graph $\PP_k$ to have vertex set $\mathcal{S}_k$ (the set of permutations of length $k$), with $\pi,\tau \in \mathcal{S}_k$ being adjacent if there exists $t$ so that $\pi(i) = \tau (i)$ for all $i\ge t$ and $\pi(i) = \tau(t-i)$ for all $i<t$ (we say such $\pi,\tau$ are \textit{related by a prefix reversal}).

An old problem (known as the \textit{pancake problem}) is to determine the growth of $f(k)$, the diameter of $\PP_k$ (which equals $\max_{\pi\in \mathcal{S}_k}\{d_{\PP_k}(\pi,\Id_k)\}$ as $\PP_k$ is vertex-transitive). It was shown by Gates and Papadimitrou that $17\lfloor k/16\rfloor  \le f(k) \le (5k+5)/3$ \cite{gates}. Later work \cite{upper,lower} has gone on to prove that
\begin{equation}
    15k/14 -O(1)\le f(k) \le 18/11 k+O(1).\label{pancakebounds}
\end{equation}
\noindent We note that the upper bound came from an intense computer-assisted proof, which involved 2220 cases.

In this paper, we study a related problem. Given $\pi,\tau\in \SS_k$, we say they are \textit{related by a suffix reversal} if there exists $t$ such that $\pi(i) = \tau(i)$ for all $i <t$ and $\pi(i) = \tau(n-i+t)$ for all $i\ge t$. We define the $\mathcal{G}_k$ to be the graph with vertex set $\SS_k$ with $\pi,\tau$ being adjacent if they are related by either a prefix reversal or suffix reversal. We shall consider $h(k)$, the diameter of $\mathcal{G}_k$ (which equals $\max_{\pi \in \SS_k}\{d_{\mathcal{G}_k}(\pi,\Id_k)\}$ due to vertex-transitivity).

Clearly, $h(k)\le f(k)$ and so by Eq.~\ref{pancakebounds} we have $h(k)\le 18k/11+O(1)$. Interestingly, the argument in \cite{lower} proving the lower bound in Eq.~\ref{pancakebounds} continues to work in this new setting, allowing us to conclude $h(k)\ge 15k/14-O(1)$ also holds. Thus, studying the growth of $h$ seems like a natural option to better understand the pancake problem. 

We manage to prove a new upper bound for $h(k)$ which improves upon the observation above that $h(k) \le 18k/11 +O(1)$.
\begin{thm}\label{main}For all $k\ge 1$, we have
\[h(k)\le 3k/2+4.\]
\end{thm}\hide{This arises from a modification of the algorithm of \cite{gates}, which worked by growing blocks of consecutive letters.}Like all past work establishing upper bounds for the pancake problem (see \cite{upper,gates}), we obtain Theorem~\ref{main} by growing blocks of consecutive letters in accordance with some potential function. Our improvement comes from initially partitioning our alphabet into pairs, and when growing our blocks we do not allow ourselves to ``split'' any pairs. This self-imposed constraint causes some of the worst case scenarios to turn out better. \hide{We note that, in proving Theoerm~\ref{main}, our usage of suffix reversals is very minimal (see Remark~\ref{suffixusage} for more details).}

For reasons discussed in Section~\ref{conc}, the coefficient $\frac{3}{2}$ is a very natural barrier. In short, showing $h(k)<(3/2-\ep)k+O(1)$ for some $\ep >0$ would require either: an innovative strategy which overcomes ``greedy local approaches'' (which currently are the only approaches used in literature), or an improvement to a variant of the so-called ``burnt pancake problem'' (defined later).

\begin{ack}We thank Bhadrachalam Chitturi and Laurent Bulteau for informing us this problem was of interest. We also thank Matt Kwan for useful suggestions about improving the presentation. Lastly, we thank Alex Bryan for help creating Figure 1.

Parts of this note were prepared while the author was at IST Austria, we thank them for their hospitality.

\end{ack}

\section{Preliminaries}\label{prelim}

For positive integer $n$, we write $[n]: =\{1,\dots,n\}$.

\hide{Next, for a finite set $\Sigma$, we write $\Sigma^*:= \bigcup_{\ell=0}^\infty \Sigma^\ell$ to denote the set of all ``finite words'' on the alphabet $\Sigma$ (with each $w\in \Sigma^\ell$ being viewed as a function from $[\ell]$ to $\Sigma$). Given $x,y\in \Sigma^*$, we write $xy$ to denote their concatenation.}

We will consider permutations $\pi \in \SS_k$ as being functions from $[k]$ to $[k]$ which are bijective.

We shall now introduce some notation and ideas, which are ported from the work of \cite{upper,gates}.

First, we note the following fact. We defer its proof to Appendix~\ref{standard}, since it is standard and well-known.
\begin{restatable}{lem}{mono}\label{mono}
For any integer $k\ge 1$, we have
\[h(k)\le h(k+1).\]
\end{restatable}\noindent This will be convenient for exposition, as our construction is best described when $k$ is divisible by $2$.

Next, given $\pi \in \SS_k$, we shall define an equivalence relation $\sim_\pi$ over $[k]$ as follows. For $i\in [k-1]$, we say that $(i,i+1)$ is a $\pi$-adjacency if $|\pi(i)-\pi(i+1)| =1$ or $\{\pi(i),\pi(i+1)\} = \{1,k\}$ (in other words, when the values differ by $\pm 1$ modulo $k$). For $j, j'\in [k]$ with $\pi^{-1}(j)\le \pi^{-1}(j')$, we say $j\sim_\pi j'$ if for all $i \in [\pi^{-1}(j),\pi^{-1}(j'))$ we have that $(i,i+1)$ is a $\pi$-adjacency.

\hide{For $i\in [k-1]$, we say that $(i,i+1)$ is a $\pi$-adjacency if $|\pi(i)-\pi(i+1)| =1$ or $\{\pi(i),\pi(i+1)\} = \{1,k\}$. We say that $j\sim_\pi j'$ if $\pi^{-1}(j),\pi^{-1}(j')$ belong to the same connected component }

The following fact was proven in \cite{gates}.
\begin{restatable}{lem}{lastblock}\label{lastblock}Let $\sigma\in \SS_k$. If $\sim_\sigma$ has only one equivalence class, then \[d_{\GG_k}(\sigma,\Id_k) \le d_{\PP_k}(\sigma,\Id_k) \le 4.\]
\end{restatable}\noindent For the convenience of the reader, a proof of this may be found in Appendix~\ref{standard}.

In Section~\ref{mainproof}, we will prove the following.
\begin{prp}\label{mainlem}Let $k =2d$. Then for every $\pi \in \SS_k$, there exists $\tau \in \SS_k$ with $\sim_\tau$ having only one equivalence class, such that
\[d_{\GG_k}(\pi,\tau) \le \frac{3}{2}k-2.\]
\end{prp}\noindent In \cite{upper,gates}, weaker forms of Proposition~\ref{mainlem} were proven when $\GG_k$ is replaced by $\PP_k$.

We now conclude by quickly deducing Theorem~\ref{main}.
\begin{proof}[Proof of Theorem~\ref{main} assuming Proposition~\ref{mainlem}]First, if $k = 2d$, then for all $\pi \in \SS_k$, we have that $d_{\GG_k}(\pi,\Id_k)\le (\frac{3}{2}k-2)+4$ by invoking Proposition~\ref{mainlem} followed by Lemma~\ref{lastblock}. So here, we have $h(k) \le \frac{3}{2}k+2$.

Otherwise, if $k = 2d-1$, then $h(k) \le h(k+1) = \frac{3}{2}(k+1)+2\le \frac{3}{2}k+4$ by Lemma~\ref{mono}.\end{proof}

\section{Proof of Proposition 2.3}\label{mainproof}

Throughout we shall assume that $k = 2d$ for some $d\ge 1$. For $j \in [k]$, we let $o_j  = (-1)^{j+1},\phi(j)=j+o_j$, so that \[\phi(1) = 2, \phi(2) = 1, \phi(3)=4,\phi(4) = 3,\dots, \phi(k-1)=k,\phi(k)=k-1.\]One should think of $j$ and $\phi(j)$ as being partners who are ``paired together''.

Now, given $\pi \in \SS_k$, we shall define a new equivalence relation $\approx_\pi$ on $[k]$, which is finer than $\sim_\pi$. Namely, given $j,j' \in [k]$, we say $j\approx_\pi j'$ if and only if either:
\begin{itemize}
    \item $j\sim_\pi j'$ but also $j\sim_\pi \phi(j)$ and $j' \sim_\pi \phi(j')$;
    \item or $j= j'$.
\end{itemize}\noindent For example, if $\pi = (2,3,4,5,1,8,6,7)$, then the set of equivalence classes for $\sim_\pi$ is \[\big\{\{2,3,4,5\},\{1,8\},\{6,7\}\big\},\] whilst the set of equivalence classes for $\approx_\pi$ is \[\big\{\{2\},\{3,4\},\{5\},\{1\},\{8\},\{6\},\{7\}\big\}.\]

The introduction of $\approx_\pi$ is our key insight. In the work of \cite{upper,gates}, when given $\pi \in \SS_k$ where $\sim_\pi$ had multiple equivalence classes, they'd give a sequence of reversals to reach some $\tau$ where $\sim_\tau$ was strictly coarser\footnote{Meaning $\sim_\tau$ has fewer equivalence classes than $\sim_\pi$, and that each equivalence class of $\sim_\pi$ is contained in some equivalence class of $\sim_\tau$.} than $\sim_\pi$. We shall do the exact same thing, but with $\approx_\pi$ in place of $\sim_\pi$.

Given permutation $\pi \in \SS_k$, we say an equivalence class $C\subset [k]$ of $\approx_\pi$ is a \textit{block} (of $\pi$) if $|C|>1$, and otherwise it is called a \textit{singleton} (of $\pi$). We say $j\in [k]$ is \textit{free} if it does not belong to a block of $\pi$ (alternatively, if $\{j\}$ is an equivalence class of $\approx_k$).

Given a permuation $\pi$, we write $S(\pi)$ to count the number singletons of $\pi$, and $B(\pi)$ to count the number of blocks of $\pi$. Furthermore, we define the potential function $\nu(\pi) = \frac{3}{2}S(\pi)+2B(\pi)$. 

We shall show the following. \begin{lem}\label{key}
For any $\pi \in \mathcal{S}_k$ with $\nu(\pi)>2$, there exists $\tau \in \mathcal{S}_k$ with \[0<  d_{\GG_k}(\pi,\tau) \le \nu(\pi)-\nu(\tau).\]\end{lem}
\hide{Now, before proving Lemma~\ref{key}, we shall quickly deduce Theorem~\ref{main}.
\begin{proof}[Proof of Theorem~\ref{main}, assuming Lemma~\ref{key}] First, consider $k=2d$
\end{proof}}
\noindent Noting that $S(\pi)+2B(\pi)\le k$ always holds (as $[k]$ is a disjoint union of $S(\pi)$ singletons and $B(\pi)$ sets of size at least two), we see that $\nu(\pi) \le 3k/2$ for all $\pi \in \mathcal{S}_k$. Meanwhile, it is clear that $\nu(\pi) = 2$ if and only if $\sim_\pi$ only has one equivalence class. Thus, iteratively applying Lemma~\ref{key} quickly gives Proposition~\ref{mainlem}.

\begin{proof}[Proof of Lemma~\ref{key}]Let $j=\pi(1)$ be the first letter of $\pi$. As $\nu(\pi)>3$, one of the four following cases will hold, and we may construct $\tau$ according to Fig~\ref{fig}.

\textbf{Case 1}: $j$ is free. In this case $\phi(j)$ will also be free. Here we use a prefix reversal to merge $j,\phi(j)$ into a block. The flipping is depicted in Fig 1a. 

\textbf{Case 2:} $j$ is not free, but $j-o_j$ (and $\phi(j-o_j)$) are free. Here we use three prefix reversals to add $j-o_j,\phi(j-o_j)$ to the block containing $j$. The sequences of flips is depicted in Fig 1bi and 1bii (depending on if $j-o_j$ is to the left of its partner $\phi(j-o_j)$).

\textbf{Case 3:} $j,j-o_j$ are not free, with $j-o_j$ being first in its block. Here we use a prefix reversal to merge the blocks containing $j$ and $j-o_j$. The flipping is depicted in Fig 1c.

\textbf{Case 4:} $j,j-o_j$ are not free, with $j-o_j$ being last in its block. Here we perform a suffix reversal followed by a prefix reversal to join the two blocks. The sequence of flips is depicted in Fig 1d.

\begin{figure}[H]
\subfloat[Case 1]{\includegraphics[width = 3in]{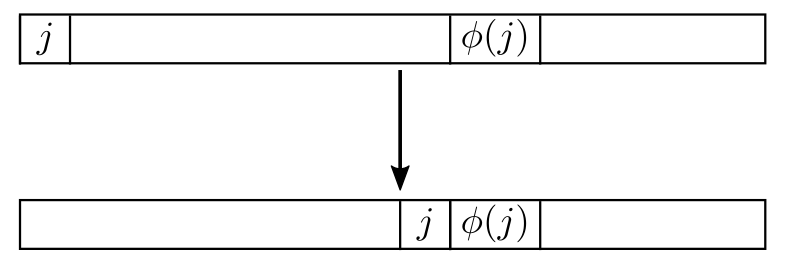}} 
\subfloat[Case 2a]{\includegraphics[width = 3in]{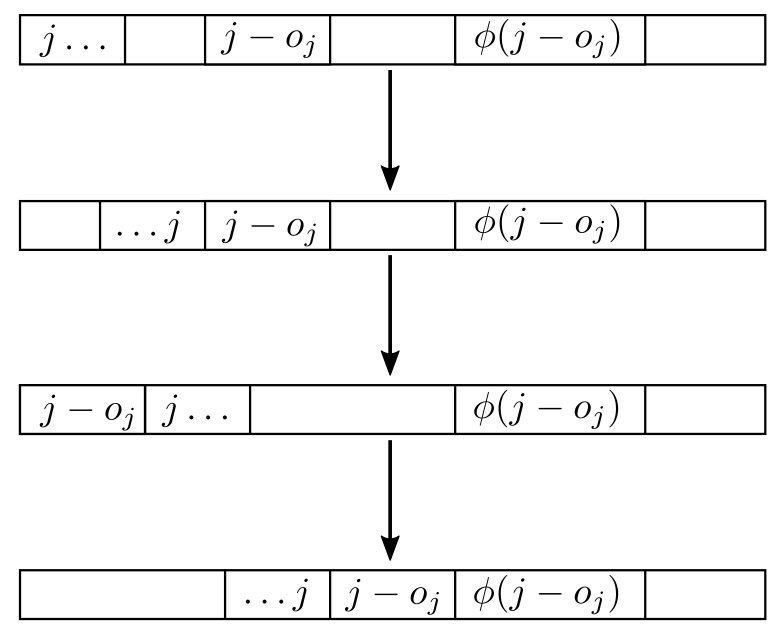}}\\
\subfloat[Case 2b]{\includegraphics[width = 3in]{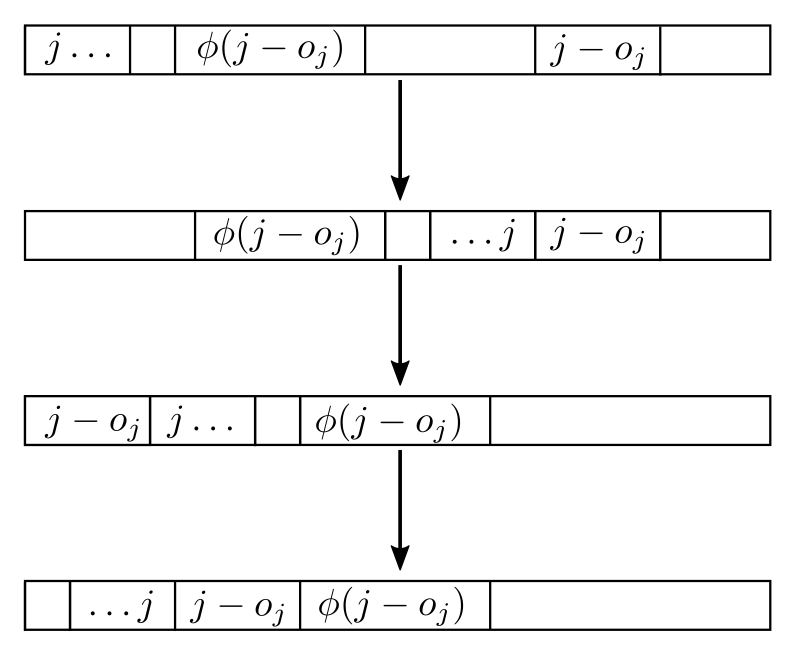}}
\subfloat[Case 3]{\includegraphics[width = 3in]{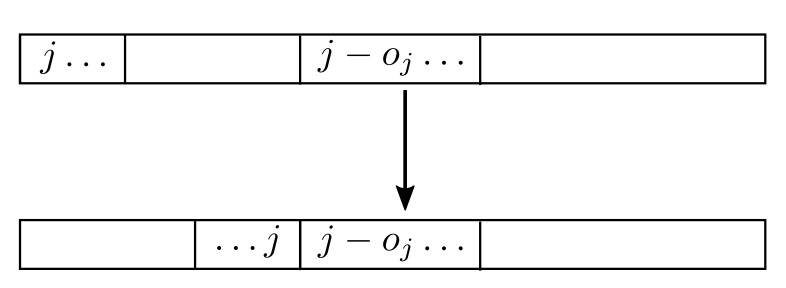}}\\
\subfloat[Case 4]{\includegraphics[width = 3in]{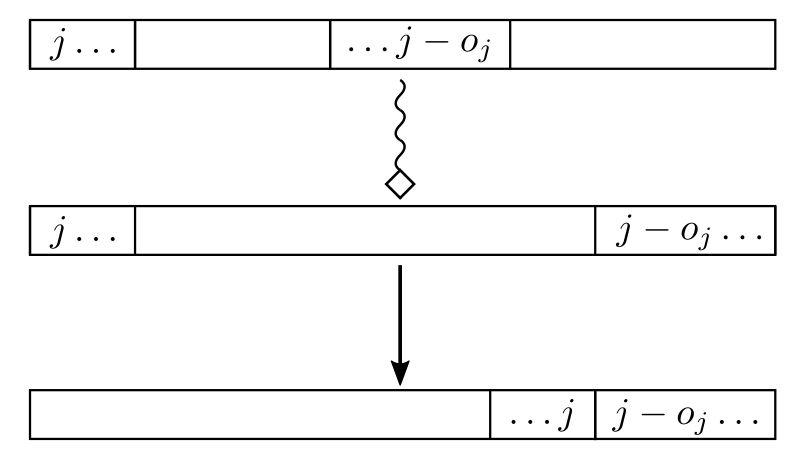}}
\caption{Our flip sequences. Here the ``wiggly arrow'' in Case 4 represents our only suffix reversal.}
\label{fig}
\end{figure}

Looking at Table~\ref{tab:caseanalysis}, we see $\tau$ behaves as desired (since the potential gain is always positive).

\begin{table}[H]
\begin{tcolorbox}[arc=0pt,boxrule=0pt]
  \centering
  \caption{Analysis of cases}
  \label{tab:caseanalysis}
  \begin{tabular}{llllll} 
    \toprule \thead{Case} &
    \thead{$\Delta(S)$\\($S(\pi)-S(\tau)$)} & \thead{$\Delta(B)$\\($S(\pi)-S(\tau)$)} & \thead{$\Delta(\nu)$\\($\nu(\pi)-\nu(\tau)=$\\$\frac{3}{2}\Delta(S)+2\Delta(B)$)} & {\thead{$d_{\PP_k}(\pi,\tau)$\\(number of flips)}} & \thead{Potential gain\\($\Delta(\nu)-d_{\PP_k}(\pi,\tau)$)}   \\
    \midrule
    Case 1 & 2 & $\ge-1$ & $\ge 1$ & 1   & $\ge0$  \\
    Case 2 & 2 & $\ge0$   & $\ge3$ & 3  & $\ge0$  \\
    Case 3 & 0 & 1   & 2 & 1 & 1  \\
    Case 4 & $\ge0$ & $\ge1$  & $\ge2$ & 2  & $\ge0$  \\
    \bottomrule
  \end{tabular}
\end{tcolorbox}
\end{table} 
\end{proof}

\hide{\begin{rmk}\label{suffixusage} We note that the only usage of a suffix reversal occurs in Case 4. Furthermore, we mention a few subcases here where we can achieve potential gain $0$. 

As in the proof above, let $j$ be the first letter of $\pi$. Suppose that $j$ is not free, with $x$ being the last letter in its block, and $j-o_j$ is not free with $y\neq j-o_j$ being the first letter in its block. This is equivalent to being in Case 4. We remark that if at least one of $x-o_x,y-o_y$ are not free, then there exists a sequence of prefix flips with potential gain 0. The first subcase ($x-o_x$ not being free) was noted in \cite{gates}, while we are not aware of reference for the second subcase. For the interested reader, we include a demonstration of both these subcases in Appendix~\ref{nosuffix}.

In light of this, it may be worthwhile investigating the pancake problem by trying to avoid the remaining subcase (where $x-o_x$ and $y-o_y$ are both free) from happening. Perhaps one can devise a potential function that negatively weights when there exists two consecutive blocks surrounded by free singletons.
\end{rmk}}

\section{Conclusion}\label{conc}

Here we briefly discuss why we believe fundamentally new ideas are required to prove $h(k) \le \left(\frac{3}{2}-\ep\right)k +O(1)$ for some $\ep>0$.

First, we recall that our paper, and all previous work establishing upper bounds on the pancake problem (cf. \cite{upper,gates}) have worked by defining a potential function $\Phi:\SS_k\to \R_{\ge 0}$ and proving a version of Lemma~\ref{key} where $\nu$ is replaced by $\Phi$. In particular, these potential functions have always been a linear combination of the number of ``singletons'' and ``blocks'' in a permutation (though our paper uses a slightly different notion of block than past literature).

Furthermore, the strategies dictated by these potential functions have always been ``locally greedy'' in the following sense. Given a vertex $\pi \in \SS_k$, these strategies ``suggest'' we take a short sequence of reversals to reach some $\pi'$ which closer to $\Id_k$ (in a sense, this is analogous to gradient descent). The suggestions from these strategies always follow two rules of thumb: 
\begin{itemize}
    \item if you can use one prefix-reversal to combine two singletons into a block, then this should be done;
    \item never use any reversal which breaks a block.
\end{itemize} Morally, when trying to find a short path from $\pi \in \SS_k$ to $\Id_k$, these rules demand that we never take certain edges (though the first rule of thumb doesn't apply when we are in the middle of executing a sequence of reversals).

Now, these locally greedy heuristics are rather intuitive. Each reversal can only create one new adjacency, and starting from a typical $\pi \in \SS_k$ this must be done at least $k-O(1)$ times to reach $\Id_k$. So it would seem quite foolish to ``pass'' on an oppurtunity to create an adjacency for the cost of one reversal, unless you know something about the global structure at this time (which seems difficult to track due to the fact that most relevant information does not get preserved by reversals). The first rule basically tells us to always ``buy'' an adjacency when it is available for the cheapest possible cost. The second rule tells us to never take an edge which decreases our adjacencies by one (i.e., never ``sell'' an adjacency and simultaneously pay another reversal for it).

\begin{rmk}
 In some instances, these rules of thumb are inoptimal. For example, starting at $(2,1,4,3)$, any path to $\Id_4$ which doesn't break blocks must use $4$ reversals. However, there is a path of length of 3 available to us, namely $(2,1,4,3) \to (4,1,2,3) \to (3,2,1,4) \to (1,2,3,4)$. This demonstrates that the second rule can be inoptimal. 
 
 Similar issues arise if we follow the first rule of thumb. Starting at $(4,1,2,3)$, we would be suggested to go to $(2,1,4,3)$, which has $\PP_4$-distance $3$ from $\Id_4$, causing us to take path of length $4$ in total. Meanwhile, there is a path of length 2 available to us, namely $(4,1,2,3)\to(3,2,1,4) \to (1,2,3,4)$.
\end{rmk}

Anyways, starting at certain permutations $\pi \in \SS_{2d}$, the first rule of thumb would force us to use $d$ steps to reach a permutation $\tau\in \SS_{2d}$ made up of $d$ blocks. This leads us to a related problem involving ``signed permutations''.

A \textit{signed permutation} is an element $(\pi,\hat{x}) \in \SS_d\times \{-1,1\}^d$. We write $\SS_d^*$ to denote the set of signed permutations. We now say that $(\pi,\hat{x}),(\tau,\hat{y})\in \SS_d^*$ are related by a prefix reversal if there exists $t$ so that $\pi(i) = \tau(i),x_i = y_i$ for $i\ge t$ and $\pi(i) = \tau(t-i),x_i = -y_{t-i}$ for $i<t$. Informally, $x_i$ as tracking the ``orientation'' of the $i$-th letter of $\pi$, and this orientation gets ``flipped'' whenever said letter is moved by a reversal. 

Define $\PP_d^*$ to be the graph on vertex set $\SS_d^*$ with vertices being adjacent if they are related by a prefix reversal. The \textit{burnt pancake problem} asks to bound $f^*(d)$, the diameter of $\PP_d^*$. The work of Gates and Papadimitrou gives that $f^*(d) \le 2d+O(1)$ \cite[Theorem~3]{gates}, and this has only been improved by additive constant (see \cite[Corollary~8.3]{cohen}).

The relevence of signed permutations, is that the subgraph of $\PP_{2d}$ induced by the set of permutations having no singletons, $\PP_{2d}[S^{-1}(0)]$, is isomorphic to $\PP_d^*$. Indeed, first observe that for $\pi \in S^{-1}(0) \subset \PP_{2d}$, the information $\pi(2),\pi(4),\dots,\pi(2d)$ is unique (since\footnote{We remind the readers that $o_j := (-1)^{j+1}$.} $\pi(2i-1) = \pi(2i)-o_{\pi(2i)}$ for $t=1,\dots,d$). The isomorphism from $\PP_{2d}[S^{-1}(0)]$ to $\PP_d^*$ is then given by taking $\varphi(\pi) = (\pi^*,\hat{x})$ so that for $i \in [d]$, $\pi^*(i) = \lceil \pi(2i)/2\rceil$ and $\hat{x}(i) = o_{\pi(2i)}$.

Recall the second rule of thumb, that we never use reversals that break blocks. Following this rule, once we have no singletons (i.e., reach a state $\pi \in S^{-1}(0)$), we maintain this. In which case, the problem reduces to finding the shortest $\PP_d^*$-path from $\varphi(\pi)$ to $\varphi(\Id_{2d})$. However, as noted above, the best known upper bound for $f^*(d)$ is $2d-O(1)$. And, after giving it some thought, we don't know how to obtain a better upper bound for $h^*(d)$ (the straightforward generalization where we allow suffix reversals of signed permutations).

Putting this together, the first rule of thumb can force us to spend $d-O(1)$ reversals to reduce to the signed case, after which we spend another $2d-O(1)$ reversals to get to the identity.

\appendix

\section{Standard pancake results}\label{standard}

Here we prove two well-known facts we use.

\mono*
\begin{proof}Consider the surjective map $\varphi:\GG_{k+1}\to \GG_k; \pi\mapsto \pi|_{\pi^{-1}([k])}$, where we ignore the placement of the letter ``$(k+1)$''. We now observe that if $\pi,\tau$ are adjacent in $\GG_{k+1}$, then we either have \begin{itemize}
    \item $\varphi(\pi) = \varphi(\tau)$;
    \item or $\varphi(\pi), \varphi(\tau)$ are adjacent in $\GG_k$
\end{itemize}(formal verification of this claim is left as an exercise to the reader). 

So, for any path $P \subset \GG_{k+1}$ with endpoints $\pi,\tau$, the image of its vertex-set $S:= \varphi(V(P))$ must be connected in $\GG_k$. And thus there is a path $P'\subset \GG_k$ from $\varphi(\pi)$ to $\varphi(\tau)$, where $V(P') \subset S = \varphi(V(P))$ (implying that the length of $P'$ is at most the length of $P$). Consequently, we have $d_{\GG_{k+1}}(\pi,\tau)\le d_{\GG_k}(\varphi(\pi),\varphi(\tau))$ for every $\pi,\tau \in \SS_{k+1}$. 

Due to the surjectivity of $\varphi$, we see that the diameter of $\GG_k$ is at most the diameter of $\GG_{k+1}$. Recalling the definiton of $h$, we get our desired result.
\end{proof}

\lastblock*
\begin{proof}As $\PP_k$ is a subgraph of $\GG_k$, the left inequality is immediate. Meanwhile, the right inequality was proven in \cite[p. 3]{gates}. We shall reproduce their proof for the reader's convenience.

By our assumption on $\sim_\sigma$, we note $\sigma$ must take one of the forms:
\begin{itemize}
    \item $\sigma = t,(t-1),\dots,1, k,(k-1),\dots,(t+1)$ for some $t\in [k]$;
    \item or $\sigma = (t+1),(t+2),\dots,k,1, 2,\dots,t$ for some $t\in [k]$.
\end{itemize}\noindent Let $A$ be the set of $\sigma$ satisfying the first bullet, and $B$ be the set of $\sigma$ satisfying the second bullet.

We first note that given any $\sigma \in A$, we can apply a $k$-letter prefix reversal to reach $\sigma' = (t+1),\dots,k,1,\dots,t$ which belongs to $B$. Hence $\max_{\sigma \in A} \{ \min_{\sigma'\in B} \{d_{\PP_k}(\sigma,\sigma')\}\} =1$.

So it now remains to show that $\max_{\sigma \in B}\{d_{\PP_k}(\sigma,\Id_k)\} \le 3$. Consider $\sigma = (t+1),\dots,k,1,\dots,t \in B$, and observe the path 
\[(t+1),\dots,k,1,\dots,t \to k,\dots,(t+1),1,\dots,t \to t,\dots,1,(t+1),\dots,k \to 1,\dots,t,(t+1),\dots,k = \Id_k.\]
\end{proof}

\end{document}